\documentclass[aihp]{imsart}

\RequirePackage{amsthm,amsmath,amsfonts,amssymb}
\RequirePackage[numbers]{natbib}
\RequirePackage[colorlinks,citecolor=blue,urlcolor=blue]{hyperref}
\RequirePackage{graphicx}
\RequirePackage{fncylab}
\RequirePackage{mathtools}
\RequirePackage[normalem]{ulem}
\RequirePackage{xcolor}
\RequirePackage{comment}

\startlocaldefs
\theoremstyle{plain}
\newtheorem*{themBJ}{Main Theorem}

\newtheorem*{them*}{Theorem}
\newtheorem{them}{Theorem}[section]
\newtheorem{pro}[them]{Proposition}

\newtheorem{lem}[them]{Lemma}

\theoremstyle{remark}
\newtheorem{defi}[them]{Definition}
\newtheorem{convention}[them]{Convention}
\newtheorem{notation}[them]{Notation} %?
\newtheorem{definotation}[them]{Definition/Notation}

\newtheorem{reminder-notation}[them]{Reminder/Notation}
\newtheorem*{problem*}{Main Problem}
\newtheorem{rem}[them]{Remark}
\newtheorem{ex}[them]{Example}    
\newtheorem{reminder}[them]{Reminder} 
\newcounter{point}
\newcounter{souspoint}[point]
\renewcommand{\thepoint}{\alph{point}}

\labelformat{point}{(#1)}          % cette commande fait appel à \usepackage{fncylab}
\labelformat{souspoint}{(#1)}  % cette commande fait appel à \usepackage{fncylab}
\newcommand{\point}{\refstepcounter{point}{\bf(\thepoint)}}

\newcommand{\maz}{\setcounter{point}{0}\setcounter{souspoint}{0}}

\newcommand{\dd}{\mathrm{d}}

\newcommand{\como}{\mathfrak{Q}}
\newcommand{\mq}{\mathfrak{MQ}}

\newcommand{\disp}{\mathfrak{Disp}}

\newcommand{\Tt}{{\mathcal T}}
\newcommand{\XX}{{\mathcal X}}
\newcommand{\EE}{{\mathcal E}}
\newcommand{\A}{{\mathcal A}}

\renewcommand{\leq}{\leqslant} 
\renewcommand{\geq}{\geqslant} 
\newcommand{\eqdef}{\coloneqq}
\newcommand{\eqdefup}{\coloneqq}

\newcommand{\eps}{\varepsilon}

\newcommand{\la}{\lambda}
\renewcommand{\P}{\mathbb{P}}
\newcommand{\p}{\mathcal{P}}
\newcommand{\m}{\mathcal{M}}
\newcommand{\ma}{\mathrm{Marg}}
\newcommand{\C}{\mathcal{C}}

\newcommand{\R}{\mathbb{R}}

\newcommand{\N}{\mathbb{N}}

\newcommand{\E}{\mathbb{E}}
\newcommand{\law}{\operatorname{Law}}

\newcommand{\pr}{\operatorname{proj}}

\newcommand{\op}{\mathopen}
\newcommand{\clo}{\mathclose}

\newcommand{\bmu}{\mbox{\mathversion{bold}$\mu$}}

\definecolor{boubelcolor}{rgb}{0.03,0.50,0.05
}

\definecolor{juilletcolor}{rgb}{.65,0.05,0
}

\DeclareUnicodeCharacter{2260}{\colorlet{memoirejuillet}{.}\color{juilletcolor}} % Le caract\`ere 
\DeclareUnicodeCharacter{00B1}{\color{memoirejuillet}{}} % Le caract\`ere   
\DeclareUnicodeCharacter{00A3}{\colorlet{memoireboubel}{.}\color{boubelcolor}} % Le caract\`ere 
\DeclareUnicodeCharacter{00A7}{\color{memoireboubel}{}} % Le caract\`ere   

\endlocaldefs

\begin{document}

\begin{frontmatter}
\title{On absolutely continuous curves in the Wasserstein space over $\R$ and their representation by an optimal Markov process}
\runtitle{Wasserstein curves and Markov representation}

\begin{aug}
\author[A]{\inits{C.}\fnms{Charles} \snm{Boubel}\ead[label=e1]{charles.boubel@math.unistra.fr}},
\author[B,C]{\inits{N.}\fnms{Nicolas} \snm{Juillet}\ead[label=e2]{nicolas.juillet@uha.fr}}
\address[A]{Institut de Recherche Math\'ematique Avanc\'ee, UMR 7501, Universit\'e de Strasbourg et CNRS, 7 rue Ren\'e Descartes, 67\,000 Strasbourg, France.
\printead{e1}}

\address[B]{Université de Haute-Alsace, IRIMAS UR 7499, 68\,100 Mulhouse, France.
\printead{e2}
}

\address[C]{Université de Strasbourg, France
}
\end{aug}

\begin{abstract}
Let $\bmu=(\mu_t)_{t\in \R}$ be a 1-parameter family of probability measures on $\R$. In \cite{boujui1} we introduced its ``Markov-quantile'' process: a process $X=(X_t)_{t\in\R}$ that resembles as much as possible the quantile process attached to $\mu$, among the Markov processes attached to $\mu$, {\em i.e.}\ whose family of marginal laws is $\bmu$.

In this article we look at the case where $\bmu$ is absolutely continuous in the Wasserstein space $\p_2(\R)$. Then $X$ is solution of a dynamical transport problem with marginals $(\mu_t)_t$. It provides a {\em Markov} minimal Lagrangian probabilistic representative of $\bmu$, which is moreover unique among the processes obtained as certain types of limits: limits for the finite dimensional topology of quantile processes where the past is made independent of the future conditionally on the present at finitely many times, or limits of processes linearly interpolating $\bmu$.

This raises new questions about ways to obtain {\em Markov} Lagrangian representatives, and to seek uniqueness properties in this framework.
\end{abstract}

\begin{abstract}[language=french]
Soit $\bmu=(\mu_t)_{t\in \R}$ une famille à un paramètre de mesures de probabilité sur $\R$. Dans \cite{boujui1} nous introduisons le processus ``Markov-quantile'' qui lui est attaché: c'est le processus $X=(X_t)_{t\in\R}$ qui ressemble le plus qu'il est possible au processus quantile associé à $\bmu$, parmi les processus markoviens associés à $\bmu$, c'est-à-dire dont la famille de marges est $\bmu$.

Dans cet article nous considérons le cas où $\bmu$ est absolument continue dans l'espace de  Wasserstein $\p_2(\R)$. Alors $X$ est solution d'un problème de transport dynamique, de marges $(\mu_t)_t$. Il fournit un représentant probabiliste lagrangien minimal {\em markovien} de $\bmu$. Il est en outre  unique parmi les processus obtenus comme certains types de limites~: limites pour la topologie de dimension finie de processus quantiles dont le passé est rendu indépendant du futur, conditionnellement au présent, en un nombre fini d'instants, ou limites de processus interpolant linéairement $\bmu$.

Ceci soulève de nouvelles questions sur les manières d'obtenir des représentants lagrangiens {\em markoviens}, et de demander des propriétés d'unicité dans ce cadre.

\end{abstract}



%\subjclass[2010]{60A10, 28A33, 60J25, 35Q35, 49J55, 49Q22}
%\begin{document}
%60A10 Probabilistic measure theory
%49Q22 Optimal transportation (AJOUTER? Nouveau MSC2020)  OUI
%60J25 Continuous-time Markov processes on general state spaces
%35Q35 PDEs in connection with fluid mechanics
%28A33 Spaces of measures, convergence of measures
%49J55 Existence of optimal solutions to problems involving randomness 

\end{frontmatter}

\section{Introduction}
In \cite{boujui1} we introduced the ``Markov-quantile'' process attached to a $1$-parameter family $\bmu=(\mu_t)_{t\in \R}$ of probability measures on $\R$. It is a process in the broad sense, {\em i.e.}\ a 1-parameter family $(X_t)_{t\in\R}$ of random variables defined on the same probability space. For the distribution of $(X_t)_{t\in\R}$ we adopted the notation $\mq((\mu_t)_{t\in \R})$, or generally simply $\mq$, that is a measure on $\R^\R$ equipped with the product $\sigma$-field. It can be called Markov-quantile measure but, by abuse of notation, we occasionally identified it with the Markov-quantile process. As usual $X_t$ may namely be chosen to be the projection on the coordinate of label $t$ for the canonical probability space $\Omega=\R^\R$ equipped with $\mq$ itself.  The Markov-quantile measure $\mq$ is characterized by the following properties:
\begin{itemize}
\item[\point\label{point:marg}] $\bmu$ is the family of its marginal laws, {\em i.e.}, for each $t$, $\mu_t$ is the law of $X_t$,
\item[\point\label{point:markov}] it is Markov,
\item[\point\label{point:ressemble}] it resembles ``as much as possible'' the quantile process $\como$ attached to $\bmu$.
\end{itemize}
The meaning of \ref{point:marg} is $\forall t\in \R, \mu_t=\mq\circ X_t^{-1}$, and that of \ref{point:markov}  is recalled in Definition \ref{defi:markovien}. Remark \ref{rem:markov} gives a practical criterion for Markov measures. The meaning of \ref{point:ressemble} is made precise in \S\ref{sec:quick}. The quantile process $\como((\mu_t)_{t\in \R})$ is the 1-parameter family $(Q_t)_{t\in\R}$ of random variables on $[0,1]$ with the Lebesgue measure, defined  by: $Q_t(\alpha)=x_{\mu_t}(\alpha)$, where $x_{\mu_t(\alpha)}$ is the quantile of $\mu_t$ of order $\alpha\in [0,1]$; see Reminder \ref{remind:quantile} for details. With this notation it is well-known that the law of  $(Q_{t_1},Q_{t_2})$ is the optimal transport plan for the quadratic cost between $\mu_{t_1}$ and $\mu_{t_2}$, as we recall in Reminder \ref{reminder:wasserstein}.

For all the details on $\mq$, we refer the reader to our initial article \cite{boujui1}, in particular its introduction and  its~\S1.5 where we give an intuition of what a Markov process that is as similar as possible to the quantile process looks like, and why it is difficult to define it. However, to avoid natural confusions some comments are in order:
\begin{itemize}
\item When the quantile process $\como$ is Markov, we have $\mq=\como$. In fact both properties are equivalent. This happens in particular for any $\bmu$ satisfying $\mu_t(x)=0$ for every $x$ and $t$ (purely non atomic measures), see \cite[Remark 1.8(a)]{boujui1}.
\item If the time index set $[0,1]$ or $\R$ is replaced by $\N$ (or a finite set $R=\{r_1,\ldots,r_n\}\subset \R$ with $r_1<\cdots<r_n$) there exists a trivial solution to our problem. The process attached to $(\mu_n)_{n\in \N}$ can be called the \emph{quantile Markov chain} and consists in the inhomogeneous Markov chain with the quantile couplings (see Reminder \ref{remind:quantile} for this notion) from $\mu_n$ to $\mu_{n+1}$ used as transition kernels, see \cite[Remark 1.16(b)]{boujui1}.
\item Our approach to define $\mq$ can be seen \emph{in very far approximation} as a ``discrete to continuous'' procedure where we use partitions $(R_n)_n$ of the time index set $[0,1]$ and the corresponding quantile Markov chains, defined as suggested in the previous point. The difficulty in \cite{boujui1} is not in extracting converging subsequences but showing that an adequate choice of the sequence of partitions enables to retain the Markov property at the limit. This last property is in fact a problem as soon as $(\mu_t)_{t\in T}$ is defined on a non discrete (but possibly still countable) $T\subset [0,1]$.
\end{itemize}

In this article, we consider $\mq$ in a more analytical context than in \cite{boujui1}, that of the dynamical optimal transport theory in duality with the continuity equation, notably in continuation with Lisini's work \cite{Lis}. We prove that $\mq$ satisfies a minimality property: the expected kinetic energy of the random curve of law $\mq$ is as small as it can be for a process that exactly interpolates $\bmu$. The novel aspect of this result is of course the Markov property. It comes with several promising questions for which we give an account later, summed up in the last section.

Now \ref{sub_preliminaires} gives  a few necessary elements for understanding and stating our Main Theorem, and \ref{sub_outline} the article's outline.

\subsection{Preliminaries and Main Theorem}\label{sub_preliminaires}

In \cite{boujui1}, we dealt with {\em any} 1-parameter family of probability measures on $\R$. In this article we consider only the —nevertheless still rich--- set of {\em continuous} curves $(\mu_t)_{t\in[0,1]}:[0,1]\to\p_2(\R)$ to the Wasserstein space over $\R$. It provides the advantage that $\como$ and $\mq$ will be identified with measures on $\mathcal{C}([0,1],\R)$ (see Notation \ref{not:mac} just below, Remark \ref{rem:concentrated} and Remark \ref{rem:sale}). The reader may already have noticed another (secondary) difference: in this article the time set is $[0,1]$.

\begin{notation}\label{not:mac} \maz
\point\ For every Polish ({\em i.e.}, complete and separable) metric space $(\mathcal{X},\mathrm{d})$ we denote by $\mathcal{C}([0,1],\mathcal{X})$ the space of continuous curves from $[0,1]$ to $\mathcal{X}$ ---or simply by $\mathcal{C}$ specially when $\mathcal{X}=\R$---, with the $\sigma$-algebra induced by the topology of \mbox{$\|\cdot\|_\infty$}. We are interested in $\p(\mathcal{C})$, that is the space of probability measures on it and we denote by $\ma_\mathcal{C}(\bmu)$ the subset $\{\Gamma\in \p(\mathcal{C}):\,\Gamma^t=\mu_t\ \text{for every }t\in [0,1]\}$ where $\Gamma^t$ is $\Gamma$ pushed forward by the map $\pr^t:\gamma\in \mathcal{C}\to \gamma(t)$. Similarly we denote by $\ma(\bmu)$ the subset of $\p\left(\mathcal{X}^{[0,1]}\right)$ defined by $\{\Gamma\in \p\left(\mathcal{X}^{[0,1]}\right):\, \pr^t_\#\Gamma=\mu_t\ \text{for every }t\in [0,1]\}$. The convergence we consider on $\p(\mathcal{C})$ is the usual weak convergence of measures used in Probability Theory, {\em i.e.,} $\Gamma_n  \to \Gamma \in \p(\mathcal{C})$ if and only if $\int f d\Gamma_n\to\int f d\Gamma$ for any bounded and continuous function $f$ defined on $\mathcal{C}$. Be cautious that the same definition, applied to the case where the measures $\Gamma_n$ and $\Gamma$ are considered in $\mathcal{X}^{[0,1]}$ endowed with the product topology is equivalent to the convergence of the finite marginals, see \cite[Reminder 1.11]{boujui1}. Both concepts can accurately be called ``weak convergence''. In the present paper to avoid confusion we call the less stringent notion of convergence ``finite dimensional convergence'' and the convergence in $\p(\C)$ ``weak convergence''.

\point\  Joint marginals of $\Gamma\in \p(\C)$ or $\Gamma \in \p\left(\mathcal{X}^{[0,1]}\right)$ on several indices are denoted by $(\pr^{r_1,\ldots,r_m})_\#\Gamma$ or $(\pr^R)_\#\Gamma$ for $R=(r_1,\ldots,r_m)$ where $\pr^{r_1,\ldots,r_m}$ is the projection map $\pr^{r_1,\ldots,r_m}: \gamma\mapsto (\gamma(r_1),\ldots,\gamma(r_m))$. We also adopt the shorthands $\Gamma^R$ and $\Gamma^{r_1,\ldots,r_n}$. In case $\pi=\Gamma^{s,t}$ with $\mu_s=\Gamma^s$ and $\mu_t=\Gamma^t$ we take the vocabulary of Optimal transport saying that $\pi$ is a \emph{transport plan} or a \emph{coupling} of $\mu_s$ and $\mu_t$. The corresponding set is denoted by $\ma(\mu_s,\mu_t)$. Generally we note $\ma((\mu_t)_{t\in\mathcal{T}})$ the set of measures with marginals $(\mu_t)_{t\in\mathcal{T}}$. Note finally that the finite dimensional convergence $\Gamma_n\to \Gamma$ of the previous paragraph writes $(\pr^{r_1,\ldots,r_m})_\#\Gamma_n\to (\pr^{r_1,\ldots,r_m})_\#\Gamma$, for every $r_1,\ldots,r_m\in [0,1]$.

\point\ For every Polish space $\XX$ we denote by $\mathcal{P}_2(\XX)$ the 2-Wasserstein space $\{\mu\in\mathcal{P}(\XX):\int \mathrm{d}(x,x_0)^2\,\dd \mu(x)<\infty\}$ over $\XX$ (here $x_0$ is some and in fact any point of $\XX$). The distance $W_2$ defined on $\p_2(\XX)$ is recalled in Reminder \ref{reminder:wasserstein}.
\end{notation}

The Markov property is a classical notion; though as it plays a central role in this article we recall its definition.

\begin{defi}[Markov measure and Markov process]\label{defi:markovien}
Let $I$ be an interval and $(X_{t})_{t\in I}$ be a process of law $\Gamma$. The measure $\Gamma$ is {\em Markov} if $X$ is a Markov process in the usual sense, for which one of the formulations is:
\begin{equation}\label{eq:def_markovien}
\forall s\in I,\forall t>s, \law(X_t|\,(X_{u})_{u\leqslant s})=\law(X_t|\,X_s),
\end{equation}
where $\law(X_t|\,(X_u)_{u\leq s})$ is the law of $X_t$ conditionally to the $\sigma$-algebra generated by the $X_u$ for $u\leq s$. (In this case \eqref{eq:def_markovien} is satisfied by any process $X'$ of law $\Gamma$).
\end{defi}

In our Main Theorem we use also the following notion, precisely built in Definition \ref{defi:quantile_discretement_markovien}. We associate, with any process measure $\Gamma$, the distribution ``$\Gamma $ made Markov at a finite set $R\subset\R$ of times'',\label{madeMark} denoted by $\Gamma_{[R]}$. For any interval $I$ disjoint of $R$, the restrictions to $I$ of (the canonical processes associated with) $\Gamma$ and $\Gamma_{[R]}$ coincide. But for any two times $s<t$ separated by at least one element of $R$, the marginals at times $s$ and $t$ are independent knowing the value of the process at any intermediate time in $R$. More generally the future of any $r\in R$, conditional upon the present, is made independent of its past. With this operation Remark \ref{rem:markov} also provides a tractable characterization of the Markov measures that is fundamental in this paper.

\begin{convention}\label{conv:croissant}When we introduce sets $\{r_1,\ldots,r_m\}$, we mean implicitly $r_1<\ldots<r_m$, if not otherwise indicated.\end{convention}

The (kinetic) energy $\EE(\gamma)$ of a mapping $\gamma:[0,1]\mapsto \mathcal{X}$ in a metric space $(\mathcal{X},\mathrm{d})$ may be introduced as follows:
\begin{equation}\label{eq:energy}
\EE:\gamma\in \mathcal{C}([0,1],\mathcal{X})\mapsto\sup_{R}\sum_{k=0}^{m}\frac{\mathrm{d}(\gamma(r_k),\gamma(r_{k+1}))^2}{r_{k+1}-r_k}\in [0,+\infty],
\end{equation}
where $R=\{r_1,\ldots,r_m\}\subset\op]0,1\clo[$ and $(r_0,r_{m+1})=(0,1)$. Actually, finite energy implies continuity: if Expression \eqref{eq:energy} is finite, then $\gamma$ is continuous (see Proposition \ref{pro:relax}\ref{item:pro:relaxA}). Furthermore, $\EE(\gamma)=\int_0^1|\dot \gamma|^2(t)dt$ in a sense that is recalled in \S\ref{sub:EetA}. This notion of energy leads to the well-known notion of action, which is central in our article:

\begin{defi}\label{defi:action} If $\Gamma\in\p((\R^d)^{[0,1]})$ is concentrated (see Remark \ref{rem:concentrated}) on $\mathcal{C}([0,1],\linebreak[1]\R^d)$ its action $\A(\Gamma)$ is defined as:
\[\A(\Gamma)=\int_{\mathcal{C}([0,1],\linebreak[1]\R^d)}\EE(\gamma)\dd \Gamma(\gamma).\]
\end{defi}

\noindent The action satisfies the following classical inequality involving energies for curves in $\R^d$ and $\p(\R^d)$; its proof will be recalled in Remark \ref{rem:about_action}:
for all $\Gamma\in\p(\mathcal{C}([0,1],\R^d))$, if $(\pr^t)_\#\Gamma\in\p_2(\R^d)$ for all $t\in [0,1]$,
\begin{equation}\label{eq:inegalite_energie}\A(\Gamma)\geq \EE((\Gamma^t)_{t\in [0,1]}),\ \ \text{where }\Gamma^t\!\eqdef\!(\pr^t)_\#\Gamma,
\end{equation}
where the distance $\mathrm{d}$ involved in the definition of $\A$, through \eqref{eq:energy}, is the Euclidean distance on $\R^d$, and the distance involved in that of $\EE$, on the right side, the Wasserstein distance $W_2$ induced by $\mathrm{d}$ on $\p_2(\R^d)$. Note that a central aspect of the present paper is the analysis of how the equality case in \eqref{eq:inegalite_energie} can occur.
\begin{defi}\label{defi:lagrangian_repr}We call ``minimal Lagrangian representative'' of $\bmu$ a measure $\Gamma\in\ma_\mathcal{C}(\bmu)$ such that \eqref{eq:inegalite_energie} is an equality.
\end{defi}

We prove two convergence results, of close types, Theorem \ref{them:action}, {\em i.e.}, our Main Theorem, and Theorem \ref{them:commun_d_et_1}, both presented in \S\ref{sec:them5}. Our Main Theorem comes as a refinement of  well-known results on minimal representatives attached to a curve $(\mu_t)_t$ that are gathered in Theorem \ref{them:known}  ; it rests on our building of the Markov-quantile process and gives the existence of a {\em Markov} minimal Lagrangian representative, which is completely new and is the main point of the present article. This gives naturally rise to the question of whether a \emph{Markov} process is unique among the minimal Lagrangian representatives. The answer is no, see Example \ref{ex:MLr_nu}. Though we state a weak result of uniqueness: it follows from the uniqueness of $\mq$ proved in \cite{boujui1} that such a Markov representative is unique among the limits of measures of quantile processes made Markov ``at a finite set of times'' in $[0,1]$.

\begin{themBJ}[{\em i.e.}, Theorem \ref{them:action}]\maz\label{them:d}
Let $\bmu=(\mu_t)_{t\in[0,1]}$ be a curve of finite energy $\EE(\bmu)<\infty$ in $\p_2(\R)$.\smallskip

\point\label{p1:them:d} {\em (Existence of a Markov minimal Lagrangian representative)} There exists a minimal Lagrangian representative $\Gamma$ for $\bmu$, {\em i.e.},  such that \eqref{eq:inegalite_energie} is an equality, namely $\int \EE(\gamma)\,\dd \Gamma(\gamma)=\EE(\bmu)$, that satisfies:
\begin{itemize}
\item $\Gamma $ is Markov,
\item there exists a nested ({\em i.e.}, increasing) sequence $(R_n)_{n\in \N}$ of finite subsets of\/ $\op]0,1\clo[$ such that $\como_{[R_n]}$ (see page \pageref{madeMark} and Definition \ref{defi:quantile_discretement_markovien} for this measure) converges to $\Gamma$ in $\p(\mathcal{C})$.
\end{itemize}

\point\label{p2:them:d} {\em(Weak uniqueness property)} If a minimal Lagrangian representative $\Gamma$ satisfies \emph{both} points of \ref{p1:them:d} then it is $\mq$.
\end{themBJ}

The following example shows that in general \emph{Markov} minimal Lagrangian representatives are not uniquely determined.
\begin{ex}[Non uniqueness for Markov minimal Lagrangian representatives, see Example 5.4 in \cite{boujui1}]\label{ex:MLr_nu}
Let $\mu_t$ be $\la_{[t-3/4,t-1/4]}+\frac12\delta_0$ (where $\lambda_{[a,b]}$ is the Lebesgue measure restricted to $[a,b]$) and $\Gamma\in \P(\mathcal{C}([0,1],\R)$ be a measure concentrated on the affine trajectories defined by $\Gamma(t\mapsto 0)=1/2$ and $\Gamma(\{t\mapsto x_0+t:x_0\in A\})=\lambda_{[-3/4,1/4]}(A)$. This measure $\Gamma$ is a Lagrangian representative of the continuity equation attached to $(\mu_t)_{t\in [0,1]}$. It is a minimizer of the action under marginal constraints. It is also Markov but it is not the Markov-quantile measure.
\end{ex}
It is an open question for us to find  properties enhancing the first point of \ref{p1:them:d}, {\em e.g.}, perhaps the \emph{strong} Markov property, or properties alternative to the second point, to make $\mq$ the unique Markov minimal Lagrangian representative. Note that \cite{boujui1} provides further characterizations of the Markov-quantile measure that are based on the stochastic orders.\smallskip

We end \S\ref{sec:them5} with Theorem \ref{them:commun_d_et_1}. Its statement is too technical to be given in this introduction. It obtains the process $\mq$ as the unique Markov limit of a sequence of processes that geodesically --in the sense of Optimal Transport, see Definition \ref{defi:disp}-- interpolate $(\mu_t)_{t\in [0,1]}$, instead of the sequence $(\como_{[R_n]})_n$ in the Main Theorem. The point is that it provides a type of construction that does not rely on the quantile process and that makes sense in $\R^d$ for any $d$ and still furnishes minimal Lagrangian representative. However, for $d>1$ it is not known whether some adequate choices in the construction can make it Markov. One of our main sources of inspiration is Lisini's paper \cite{Lis} whose results are similar though our construction differs from it in several points, see Remark \ref{rem:lisini}. Lisini uses a sequence of dyadic partitions to attach a minimal Lagrangian representative to each absolutely continuous of order $p>1$ curve $(\mu_t)_{t\in [0,1]}$ of measures on metric spaces $\mathcal X$ that are more general than $\R^d$. 

We stress that the Markov property was up to now not involved  in the a priori rather analytic context of the dynamical Optimal Transport. As explained in \S1.3 of \cite{boujui1}, we came to involve it while we were considering Kellerer's Theorem \cite{Ke72}, that is nowadays mostly represented in Martingale Optimal Transport (and Peacocks), a young subfield of Optimal Transport that takes advantage of the older tradition of ``classical'' Optimal Transport. We found it particularly interesting to bring the other way around with the Markov property a new ingredient back to the parent theory.\smallskip

\subsection{Outline of the article}\label{sub_outline}
 In \S\ref{sec:intro5} we give a brief historical overview of the set of problems in which our results take place; this introduces the main concepts at stake and motivates our work. In \S\ref{sec:quick} we gather the few elements of \cite{boujui1} on which the present work relies, and that are necessary to its understanding. Theorems \ref{them:known} and \ref{them:from} are good summaries of the these two prelininary sections. In \S\ref{sec:them5} we state and prove Theorem \ref{them:action}, that is our Main Theorem above, and Theorem \ref{them:commun_d_et_1}. Finally \S\ref{sec:unique} presents some open questions raised by our 1-dimensional result.

\section{State of the art}\label{sec:intro5}
As we briefly mentioned at the beginning of the introduction and explain below in Reminder \ref{reminder:wasserstein}, quantile couplings are optimal transport plans for the quadratic cost function. This suggests that the quantile process $\como$ is a minimizer for some dynamical optimal transport problems. This is true and rather well-known ; one approach is in \cite{Pass} (see also \cite{BeGr}).
In this section we recall another standard approach that formulates optimality in terms of minimal Lagrangian representatives. Subsection \S\ref{sub:framework} explains the framework and concludes with Theorem \ref{them:known}. In \S\ref{sub:EetA} we prepare the following with useful definitions and results on $\EE$ and $\A$.

\subsection{Historical framework and reminders on minimal representatives}\label{sub:framework}

The origin of this research goes back to the interpretation by Arnold in \cite{Arn} of the solutions of the incompressible Euler equations on a compact Riemannian manifold as geodesic curves in the space of diffeomorphisms preserving the volume. In \cite{Brenier1989}, Brenier relaxed the minimization problem attached to those geodesics and introduced \emph{generalized geodesics} that are, in probabilistic terms, continuous processes $X=(X_t)_{t\in[0,1]}$ with $\law(X_t)$ equal to the Riemannian volume at every time%$\law(X_t)=\vol$ at every time, where $\mathrm{Vol}$ denotes the Riemannian volume
. The quantity to minimize is the action $\A(X)=\E\int_{0}^1|\dot{X_t}|^2\,\dd t=\int_{0}^1\E|\dot{X_t}|^2\,\dd t$, under the constraint that the marginals $\law(X_t)$ and $\law(X_0,X_1)$ are prescribed. Later, see \cite{JKO, Ott}, Otto and his coauthors discovered that the solutions of some PDEs, in particular the Fokker--Planck and porous medium equations can be thought of as curves of maximal (negative) slope for some entropy functionals in the space of probability measures $\p_2(\R^d)$ endowed with the 2-transport distance (alias Wasserstein distance). It catches a comprehensive picture of the infinite dimensional manifold of measures used in optimal transport, building a differential calculus on it, called ``Otto calculus''. In this context, the derivative of the curve $(\mu_t)_t$ at time $t$ shall be seen as a vector field $v_t$ of gradient type, square integrable with respect to $\mu_t$, such that the continuity equation:
\begin{equation}\label{eq:continuity_intro}\frac{\dd}{\dd t}\mu_t+\mathrm{div}(\mu_t v_t)=0
\end{equation}
is satisfied. This special (non homogeneous) vector field is precisely the minimizer of $\iint |v_t|^2\,\dd\mu_t \dd t$ among the vector fields satisfying \eqref{eq:continuity_intro}, the corresponding value being $\EE(\bmu)$. A thorough study of those questions has been conducted in the monograph \cite{AGS} by Ambrosio, Gigli and Savar\'e (see also \cite{Bernard, Mani, AS}) under very loose assumptions on the curve $(\mu_t)_t$ and the vector field $(v_t)_t$. They proved, in particular, that the vector field $(v_t)_t$ of minimal energy is uniquely determined if $(\mu_t)_t$ is absolutely continuous of order 2 (see ``$AC^2$'' in \S\ref{sub:EetA}). They showed also that a process minimizing the action, for prescribed marginals $\mu_t$, exists, by using limits of solutions of mollified versions of \eqref{eq:continuity_intro}. Almost every trajectory of this process is in fact solution of the Cauchy problem $\dot{X}_t=v_t(X_t)$. Furthermore, all the minimal Lagrangian representatives $(X_t)_t$  are tangent to the minimizing vector field $v=(v_t)_t$ attached to $\bmu$. Note that in the case of smooth enough curves $(\mu_t)_t$, the vector field $(v_t)_t$ is also smooth and the minimal Lagrangian representative is uniquely determined. But in general whereas this field $v$ is unique, no uniqueness statement is satisfied by $(X_t)_t$. In a further work \cite{Lis}, Lisini studied, in fact in a broader framework, the $AC^2$ curves of probability measures on a metric space. In this context where the continuity equation is not defined, he also proved that there exists a minimal Lagrangian representative. The following is standard and based on the works of Ambrosio--Gigli--Savar\'e and Lisini.

\begin{them}[Existence and uniqueness for minimal representatives]\label{them:known}\maz
Take a curve $\bmu=\linebreak[1](\mu_t)_{t\in[0,1]}$ in Wasserstein space $\mathcal{P}_2(\R^d)$ with finite energy $\EE(\bmu)$. Then:\smallskip

\point\label{pt:eul} {\em (Eulerian statement)} There exists a family $(v_t)_{t\in [0,1]}$ of vector fields satisfying the continuity equation \eqref{eq:continuity_intro} and such that the inequality:
\[\int_0^1\int|v_t|^2\,\dd \mu_t\,\dd t\geqslant \EE(\bmu)\]
becomes an equality. This family is unique.\smallskip

\point\label{pt:lag} {\em (Lagrangian statement)} There exists $\Gamma\in\ma_\mathcal{C}(\bmu)$ such that Inequality \eqref{eq:inegalite_energie}:
$\A(\Gamma)\geqslant \EE(\bmu)$
is an equality.\smallskip

\point\label{pt:ens} {\em (Link between them)} For any $\Gamma$ minimizing the action, {\em i.e.}, making \eqref{eq:inegalite_energie} an equality, the curve $\gamma\in \mathcal{C}$ is $\Gamma$-almost surely a solution of the ODE:
$$\dot\gamma(t)=v_t(\gamma_t),$$
for almost every time. 
\end{them}

The uniqueness of $(v_t)_{t\in [0,1]}$ in \ref{pt:eul} encourages to seek, among the processes $(X_t)_t$ in \ref{pt:lag}, processes satisfying some additional properties, trying by that to yield uniqueness. This is what we do in Theorems \ref{them:action} and \ref{them:commun_d_et_1} with $\mq$.

\subsection{Some reminders on the energy ${\cal E}$ of a curve of probabilities and the action ${\cal A}$ of a probability on curves.}\label{sub:EetA}

The definitions and results on curves, their energy and the Wasserstein distance recalled here are close to Brenier's paper \cite[Section 3]{Brenier1989}. These reminders are required to prove later that $\como$ and $\mq$ are minimal Lagrangian representatives, what is done respectively in Remark \ref{rem:about_action} and Section \ref{sec:them5}, the former being rather basic the latter being new.

\begin{reminder-notation}\label{nota:length}\maz\ Let $(\mathcal{X},\mathrm{d})$ be a  metric space and $\gamma$ a curve in $\mathcal{C}([0,1],\mathcal{X})$. The curve $\gamma$ is said to be absolutely continuous of order $p\geq 1$ and we note $\gamma\in AC^p([0,1],\mathcal{X})$ (or simply $AC^p$) if there exists $m\in L^p([0,1], \R)$ such that $d(\gamma(a),\gamma(b))\leq \int_a^b m\,\dd\la$ for every $a<b$. If $\gamma\in AC^p$, an admissible choice for $m$ is the so-called metric derivative $|\dot\gamma|$ defined for almost every $t$ by:
\[|\dot\gamma|(t)=\lim_{h\to 0}\frac{d(\gamma(t+h),\gamma(t))}{h}.\]
(if $(\mathcal{X},\mathrm{d})=(\R^n,\|\cdot\|)$ and $\gamma$ is differentiable at $t$, this is $\|\dot\gamma(t)\|$, so the notation is consistent). 
 \end{reminder-notation}

Recall that $\EE(\gamma)$ was introduced in \eqref{eq:energy} for a curve $\gamma$ parametrized on $[0,1]$. The definition extends trivially for curves on $[a,b]$. A {\em partition} of an interval $[a,b]$ is a finite subset $R=\{r_0,\ldots,r_{m+1}\}$ of $[a,b]$ with $(r_0,r_{m+1})=(a,b)$. The {\em mesh} $|R|$ of $R$ is $\max_{k=0}^m|r_{k+1}-r_k|$. We denote by $\EE(\gamma,R)$ the quantity approximating $\EE(\gamma)$ on the right-hand side in \eqref{eq:energy}. The next proposition gathers well-known facts on $\EE$.% (but see \emph{e.g.}, \cite[Lemma 4.3]{Vi2} for \ref{ccc}).

\begin{pro}\label{pro:relax}\maz
Let $\gamma$ be a mapping from $[a,b]$ to $\mathcal{X}$. Then:\smallskip

\point\label{item:pro:relaxA} If $\EE(\gamma)<\infty$ then $\gamma$ is continuous.

\point\label{item:pro:relaxB}\  {\bf(i)} If a partition $R'$ is finer than $R$, $\EE(\gamma,R)\leq\EE(\gamma,R')$. {\bf(ii)} If $\gamma$ is continuous, the limit $\lim_{|R|\to 0} \EE(\gamma,R)$ is well-defined and equals $\EE(\gamma)$. {\bf(iii)} $\EE(\gamma)$ is finite if and only if $\gamma\in AC^2([a,b],\mathcal{X})$; in this case $\EE(\gamma)=\int_a^b|\dot{\gamma}|^2(t)\dd t$.

\point\   \label{ccc}$\EE(\gamma)$ is lower semi-continuous for the  uniform convergence.
\end{pro}

Remark \ref{rem:about_action} recalls properties of $\A$ introduced in Definition \ref{defi:action}. Its points {\bf(c, d)} use Reminders \ref{remind:quantile} and \ref{reminder:wasserstein}.

\begin{reminder}[Quantiles]\label{remind:quantile}
The {\em quantile of level $\alpha$} of a measure $\mu\in\m(\R)$ is the smallest real number $x_\mu(\alpha)$ such that $\mu(\op]-\infty,\linebreak[1]x_\mu(\alpha)])\geq \alpha$ and $\mu([x_\mu(\alpha),+\infty\clo[)\geq 1-\alpha$. 
The quantile process $(Q_\tau)_{\tau\in \Tt}$, defined on $\Omega=[0,1]$ with the Lebesgue measure, is given by $Q_t(\alpha)=x_{\mu_t}(\alpha)$, and we denote $\law(Q)$ by $\como\in \ma((\mu_t)_{t\in \Tt})$. In particular $\law(Q_t)=\mu_t$ for every $t\in \Tt$. See Definition 3.23 of \cite{boujui1} for full details. When $\Tt$ has cardinal 2, $\como$ is called the \emph{quantile transport (plan)} or the \emph{quantile coupling} (it is a slight abuse since couplings are \emph{usually} random variables).
\end{reminder}

\begin{reminder}[Optimal transport]\label{reminder:wasserstein}
On $\p(\R^d)^2$ the following infimum (minimum by the Prokhorov Theorem) has all the properties of a distance except that it may be infinite; it is called the $2$-Wasserstein distance:
\begin{equation}\label{eq:def_w}
W_2(\mu,\nu)=\min_{P\in \ma(\mu,\nu)}\sqrt{\int \|y-x\|^2\dd P(x,y)}.
\end{equation}
On the Wasserstein space (recall Notation \ref{not:mac}), $W_2$ is finite, thus is a true distance. A minimizer $P$ of \eqref{eq:def_w} is called an \emph{optimal} transport plan between $\mu$ and $\nu$. If $d=1$ and $W_2(\mu,\nu)<\infty$ the quantile coupling $\como(\mu,\nu)$ introduced in Reminder \ref{remind:quantile} is the unique optimal transport plan, see for instance \cite{RR1}. Therefore, for the quantile process $\como\in\ma((\mu_t)_t)$:
\begin{equation}\label{eq:q_et_w}
W_2(\mu_s,\mu_t)=\sqrt{\int |y-x|^2\,\dd \como^{s,t}(x,y)}.
\end{equation}
\end{reminder}

\begin{rem}\label{rem:about_action} {\bf(a)} If $\A(\Gamma)<+\infty$, $\Gamma$ is in fact concentrated on $AC^2$.

{\bf(b)} If $\Gamma$ is a measure on $\mathcal{C}$, \emph{e.g.}, an element of $\ma_\mathcal{C}(\bmu)$, then:
\begin{align}\label{eq:action_energy}
\A(\Gamma)&\eqdef\int_{\mathcal{C}}\lim_{|R|\to 0}\EE(\gamma,R)\,\dd\Gamma(\gamma)=\lim_{|R|\to 0}\int_{\mathcal{C}}\EE(\gamma,R)\,\dd\Gamma(\gamma)
\end{align}
because of the monotone convergence theorem: use a monotone sequence of partitions and Proposition \ref{pro:relax}\ref{item:pro:relaxB}.

{\bf (c)} If $\Gamma\in\ma_\mathcal{C}(\bmu)$, then:
\begin{equation}\label{eq:a_superieur_a_e}
\A(\Gamma)\geqslant \EE(\bmu).
\end{equation}
Indeed:
\begin{align}
\int_{\mathcal{C}}\EE(\gamma,R)\,\dd\Gamma(\gamma)&=\int_\mathcal{C} \sum_{k=1}^m\|\gamma(r_{k})-\gamma(r_{k+1})\|^2/(r_{k+1}-r_k)\,\dd \Gamma(\gamma)\nonumber\\\label{ici}
&= \sum_{k=1}^m \left(\int_\mathcal{C} \|\gamma(r_{k})-\gamma(r_{k+1})\|^2/(r_{k+1}-r_k)\,\dd \Gamma(\gamma)\right)\\ &\geqslant \sum_{k=1}^m W_2(\mu_{r_{k}},\mu_{r_{k+1}})^2/(r_{k+1}-r_k)=\EE(\bmu,R). \nonumber
\end{align}
The inequality comes from the fact that $(\pr^{r_k,r_{k+1}})_\#\Gamma$ is in $\ma(\mu_{r_k},\mu_{r_{k+1}})$, so that $\int_\mathcal{C} \|\gamma(r_k)-\gamma(r_{k+1})\|^2\,\dd \Gamma(\gamma)\geqslant W_2(\mu_{r_k},\mu_{r_{k+1}})^2$. Now, thanks to (\ref{eq:action_energy}), when $|R|$ tends to 0 this provides $\A(\Gamma)\geqslant \EE(\bmu)$.

{\bf(d)} In dimension 1 if $\EE(\bmu)<+\infty$ then $\como$ is a minimal Lagrangian representative: by \eqref{eq:q_et_w} endowed in (c), equality occurs in \eqref{ici} for $\Gamma=\como$, thus $\A(\como)=\EE(\bmu)$. We used that $\como$ is concentrated on $\C$. This is discussed in Remark \ref{rem:sale}.
\end{rem}

\maz

\section{The Markov-quantile process $\mq$ attached to $\bmu$}
\label{sec:quick}

We gather below the main notions of \cite{boujui1} the present article relies on like concatanation (Definition \ref{defi:concatenation}) and measure made Markov at the times of a partition (Definition \ref{defi:quantile_discretement_markovien}). Theorem \ref{them:from} that concludes the section is the core of the theorems in \cite{boujui1}. However, let us start with an important measure theoretic remark.

\begin{rem}\label{rem:concentrated} As will deal with measures in $\ma_\mathcal{C}((\mu_t)_t)$, but make use of theorems about $\ma((\mu_t)_t)$, we wish to see $\ma_\mathcal{C}((\mu_t)_t)$ as a subset of $\ma((\mu_t)_t)$, {\em i.e.}, to give a meaning to the subset ``$\{Q\in\ma((\mu_t)_t)\,|\,Q(\mathcal{C}([0,1],\R^d))=1\}$'', which makes no sense as $\mathcal{C}([0,1],\R^d)$ is not in the cylindrical $\sigma$-algebra of $(\R^d)^{[0,1]}$. It is classically done as follows. For any $Q\in\ma((\mu_t)_t)$, we will say that $Q$ is ``{\em concentrated on $\C$}'' if, for any dense countable subset $D$ of $[0,1]$, $Q(\{f\in(\R^d)^{[0,1]}\,|\,f_{|D}\ \text{is uniformly continuous}\})=1$; the latter subset is in the cylindrical $\sigma$-algebra of $(\R^d)^{[0,1]}$, as it is a countable union of countable intersections of open sets of the product topology. Notice that the uniform continuity condition amounts to the fact that $f_{|D}$ extends as a continuous function on $[0,1]$. Then $\ma_\mathcal{C}((\mu_t)_t)$ and the set of measures of $\ma((\mu_t)_t)$ concentrated on $\C$ are in 1-1 correspondence as follows.

-- If $\Gamma\in\ma_\mathcal{C}((\mu_t)_t)$, you can define $\hat\Gamma\in\ma((\mu_t)_t)$ concentrated on $\C$, by $\hat\Gamma:B\mapsto\Gamma(B\cap\mathcal{C}([0,1],\R^d))$.

-- If $Q\in\ma((\mu_t)_t)$, take any (its choice will not matter) countable dense subset $D$ of $[0,1]$ and define $\check Q\in\ma_\mathcal{C}((\mu_t)_t)$ by $\check Q:B\mapsto Q(\{f\in(\R^d)^{[0,1]}\,|\,\exists g\in B:f_{|D}=g_{|D}\})$. We let the reader check that, as $D$ is countable, the latter subset is in the cylindrical $\sigma$-algebra and that, in restriction to the the set of measures concentrated on $\C$, the definition of $\check Q$ is independent of the choice of $D$, $Q\mapsto \check Q$ is injective, and $\Gamma\mapsto\hat\Gamma$ is its inverse function.

So by a slight abuse, we will not distinguish $\Gamma$ and $\hat\Gamma$ or $Q$ and $\check Q$. For $\Gamma\in\ma_\mathcal{C}((\mu_t)_t)$ and $R$ a finite subset of $\R$, this gives sense, \emph{e.g.},  to $\Gamma_{[R]}$ after Definition \ref{defi:quantile_discretement_markovien}.
\end{rem}

\begin{rem}\label{rem:sale}
In Remark \ref{rem:about_action}(d) we used \eqref{eq:action_energy} voluntarily without justification to simplify the purpose. In fact since we don't know whether $t\in [0,1]\mapsto Q_t(\alpha)$ (remind Reminder \ref{remind:quantile}) is continuous for almost every $\alpha\in [0,1]$ we need to prove that $\como$ is concentrated on $\C$ in the sense of Remark \ref{rem:concentrated}. In a nutshell this can be shown as follows: if $\EE(\bmu)<+\infty$, for every increasing sequence of partitions $(R_n)_n$ with $R_\infty=\bigcup_n R_n$ dense in $[0,1]$ we have $\lim_{n\to\infty} \int_{\R^{[0,1]}} \EE(\gamma,R_n)\mathrm d\como(\gamma)\leq \EE(\bmu)<+\infty$. Therefore, by the monotone convergence theorem, $\gamma_{|R_{\infty}}$ has $\como$-almost surely finite energy as a mapping defined on $R_\infty$. Since $R_\infty$ is arbitrary chosen, this suffices to prove that $\como$ is concentrated on $\C$ as defined in Remark \ref{rem:concentrated}.
\end{rem}

Now $E$ stands for some Polish space and $\mathcal{B}(E)$ for the set of its Borel subsets.

\begin{definotation}\label{defnot}
A probability kernel, or kernel $k$ from $E$ to $E'$ is a map $k:E\times \mathcal{B}(E')\to [0,1]$ such that $k(x,\cdot)$ is a probability measure on $E'$ for every $x$ in $E$ and $k(\cdot,B)$ is a measurable map for every $B\in {\mathcal B}(E')$.

Every transport plan $P\in\p(E\times E')$ can be disintegrated with respect to its first marginal $P^1\eqdefup(\pr^1)_\#P$ and a kernel that we denote by $k_P$, defined from $E$ to $E'$, so that, for every bounded continuous function $f$:
$$\iint f(x,y)\,\dd P(x,y)=\int\left(\int f(x,y)\,k_P(x,\dd y)\right)\,\dd P^1(x)$$

\end{definotation}

The two following concepts may appear unusual. The interested reader is invited to consult \cite{boujui1} for more details.

\begin{defi}[See \cite{boujui1}, Definition 2.8] \label{defi:concatenation} 
If $\mu_i\in\p(E_i)$ for $i\in\{1,2,3\}$, $P_{1,2}\in\ma(\mu_1,\mu_2)$ and $P_{2,3}\in\ma(\mu_2,\mu_3)$, their {\em concatenation} $P_{1,2}\circ P_{2,3}$ is the unique $R\in\p(E_1\times E_2\times E_3)$ such that for every $ (B_1,B_2,B_3)\in\prod_{i=1}^3{\cal B}(E_i)$:%{\cal B}(E_1)\times{\cal B}(E_2)\times{\cal B}(E_3)$:
\begin{align}\label{eq:concatenation}
R(B_1\times B_2\times B_3)&=\int_{x\in B_1}\int_{y\in B_2}\int_{z\in B_3} \dd\mu_1(x)k_{1,2}(x,\dd y)k_{2,3}(y,\dd z).
\end{align}
In particular, $R\in\ma(\mu_1,\linebreak[1]\mu_2,\linebreak[1]\mu_3)$, $(\pr^{1,2})_\#R=P_{1,2}$, and $(\pr^{2,3})_\#R=P_{2,3}$.
\end{defi}

\begin{defi}[See Definition 4.18 of \cite{boujui1}] \label{defi:quantile_discretement_markovien}  If $\Gamma \in\ma((\mu_t)_t)$ and if $R=\{r_1,\ldots,r_m\}\subset\R$ we denote by $\Gamma_{[R]}\in\ma((\mu_t)_t)$ the measure \emph{$\Gamma$ made Markov at the points of $R$} defined by the data of its finite marginals $(\pr^{S})_\# \Gamma_{[R]}$, for all finite $S$ containing $R$, as follows.
$$(\pr^{S})_\# \Gamma_{[R]}= \underbrace{\Gamma^{s^0_1,\ldots, s^0_{n_0},r_1}\circ \Gamma^{r_1,s^1_{1},\ldots,s^1_{n_1},r_2}\circ\cdots\circ \Gamma^{r_m,s^m_{1},\ldots,s^m_{n_m}}}_{\text{(denoted immediately below by $\Gamma_S$)}},$$
where $S=\{s^0_1,\ldots, s^0_{n_0},r_1,s^1_{1},\ldots,s^1_{n_1},r_2,\ldots,r_m,s^m_{1},\ldots,s^m_{n_m}\}$ and where the first or last term disappears if $n_0$ or $n_m$ is null, respectively. These marginals are consistent in the sense that for all finite subsets $S$ and $S'$ of $\R$, containing $R$, $S'\subset S\Rightarrow (\pr^{S'})_\#\Gamma_S=\Gamma_{S'}$. So by the Kolomogorov-Daniell theorem (see Proposition 2.12 of \cite{boujui1}), this defines  $\Gamma_{[R]}$. We also commit an abuse of language: $\Gamma_{[R]}$ is rather the ``\emph{law of a process $X$ of law} $\Gamma$, made Markov at the points of $R$''.
\end{defi}

\begin{rem}\label{rem:markov}Let $I$ be some interval. A process $X=(X_t)_{t\in I}$ and $\Gamma\in\p(\R^I)$ its measure; $X$ is therefore Markov (see Definition \ref{defi:markovien}) if and only if, for any finite subset $R$ of $I$, $\Gamma_{[R]}=\Gamma$.
\end{rem}
\begin{rem}
Note that if $\Gamma$ is concentrated on $\ma_\C(\bmu)$ then $\Gamma_{[R]}$ is also concentrated on $\ma_\C(\bmu)$.
\end{rem}

Here are the parts of Theorems A and B of \cite{boujui1} that are used in this article.

\begin{them}[From the main theorems in \cite{boujui1}]\label{them:from}
There exists one and only one Markov law $\mq$ that is a limit in the finite-dimensional sense of sequences of laws of type $(\como_{[R_n]})$, being $(R_n)$ an increasing sequence of finite subsets of $\R$. Moreover, one can assume that $R_\infty=\cup_n R_n$ is dense in $\R$.

\end{them}
\begin{proof}
The existence of such an increasing sequence $(R_n)_{n\in \N}$ such that $\como_{[R_n]}$ converges to $\mq$ in the finite-dimensional sense comes from \cite[Theorem B]{boujui1} (where the finite-dimensional convergence is called \emph{weak} convergence). The uniqueness comes from the uniqueness of $\mq$ as a Markov measure satisfying (iv) in \cite[Theorem A]{boujui1}. The density statement comes from (c)(i) in \cite[Theorem 4.21]{boujui1} that is a more detailed version of Theorem B.
\end{proof}
\begin{rem}
The claim of page \pageref{point:ressemble} that $\mq$ resembles as much as possible the quantile process $\como$ attached to $\bmu$ clearly appears in Theorem \ref{them:from}. It is also strengthened from the side of the stochastic orders by (a)(iii) of Theorem A in \cite{boujui1}.
\end{rem}

\section{Our resulting theorems on $\mq$ as a minimizer in this context}\label{sec:them5}

In this section we state and prove our theorems. In Lemma \ref{lem:energy_rendu_markov} and Proposition \ref{pro:ap_egale_emu} we pursue our investigation on $\como$ started in Remark \ref{rem:about_action} with new results on $\como_{[R]}$ and $\mq$, respectively. Then we prove Theorem \ref{them:action} and Theorem \ref{them:commun_d_et_1}.
\begin{lem}[$\como$ and $\como_{[R]}$ are minimal Lagrangian representatives]\label{lem:energy_rendu_markov}
Let $\bmu=(\mu_t)_{t\in[0,1]}$ be a family of real measures in $\p_2(\R)$ and $\como$ the attached quantile process. We assume that $\como$ is concentrated on $\mathcal{C}$ so that $\A(\como)$ makes sense (this happens as soon as $\EE(\bmu)<\infty$, recall Remark \ref{rem:sale}). Let $R$ be a partition of $[0,1]$. Then $\A(\como_{[R]})=\A(\como)\left(=\EE(\bmu)\in [0,+\infty]\right)$.
\end{lem}
\begin{proof}
The equality $\A(\como)=\EE(\bmu)$ is part of Remark \ref{rem:about_action}. In fact the other equality $\A(\Gamma)=\A(\Gamma_{[R]})$ is satisfied not only for $\Gamma=\como$ but for any $\Gamma$ concentrated on $\mathcal{C}([0,1],\R^d)$ even for $d>1$. Please look at Remark  \ref{rem:about_action} and consider the equality in \eqref{ici} to see that for partitions $(R_n)_{n\in \N}$ finer than $R$ one has $\int \EE(\gamma, R_n)d\Gamma(\gamma)=\int \EE (\gamma, R_n)d\Gamma_{[R]}(\gamma)$. For a sequence of such partitions, by \eqref{eq:action_energy} one gets:
\[\A(\Gamma_{[R]})=\lim_{|R_n|\to 0}\int \EE(\gamma, R_n)d\Gamma_{[R]}(\gamma)=\lim_{|R_n|\to 0}\int \EE (\gamma, R_n)d\Gamma(\gamma)=\A(\Gamma).\qedhere\]

\end{proof}

Lemma \ref{lem:energy_rendu_markov} ``passes to the (finite dimensional) limit'' when $(R_n)_n$ is such that $\como_{[R_n]}(\bmu)\underset{n\rightarrow\infty}{\longrightarrow} P$, where $P\in\ma_{{\cal C}}(\bmu)$ coincides with the Markov-quantile measure $\mq$ (in the sense of Remark \ref{rem:concentrated}). Recall that, for simplicity, depending on the context we see $\mq$ (or $\como$) as an element of $\ma_\mathcal{C}(\bmu)\subset \p(\mathcal{C})$ or $\ma(\bmu)\subset \p(\R^{[0,1]})$.

\begin{pro}\label{pro:ap_egale_emu}
The Markov-quantile process $\mq\in \ma(\bmu)$ satisfies $\A(\mq)=\EE(\bmu)$. Moreover for every $(R_n)_n$ as in Theorem \ref{them:from}, $(\como_{[R_n]})_n$ converges weakly to $\mq$ in $\ma_\mathcal{C}(\bmu)\subset\p(\mathcal{C})$.
\end{pro}
\begin{proof}
Let $(R_n)_n$ be a sequence of partitions of $[0,1]$ such that $\como_{[R_n]}\rightarrow\mq$ in the finite-dimensional sense as in Theorem \ref{them:from}. To get the result, it suffices to recall that $\A$ is lower semi-continuous (Proposition \ref{pro:relax}\ref{ccc}) and that it is known to have compact sublevels in the weak topology, see \cite[Proof of Theorem 3.3]{AmFi09}. With Lemma \ref{lem:energy_rendu_markov}, it implies that any subsequence $s$ of $(\como_{[R_n]})_n$ admits a (weak, and hence finite-dimensional) limit point $\Gamma_s$. By uniqueness $\Gamma_s$ is always $\mq$ so that we have proved $\como_{[R_n]}\rightarrow\mq$ weakly. Hence ${\cal A}(\mq)\leqslant \A(\como_{[R_n]})=\EE(\bmu)$, so by \eqref{eq:a_superieur_a_e}, ${\cal A}(\mq)=\EE(\bmu)$.
\end{proof}

Here is our Main Theorem. Notice that by Theorem \ref{them:known}\ref{pt:ens}, the random curves of the Markov-quantile process are integral curves of the minimizing vector field in  Theorem \ref{them:known}\ref{pt:eul}.

\begin{them}[$\mq$ is a Markov minimal Lagrangian representative]\label{them:action}\maz
Take a curve $\bmu=\linebreak[1](\mu_t)_{t\in[0,1]}$ in Wasserstein space $\mathcal{P}_2(\R)$ with finite energy $\EE(\bmu)$. There exists $\Gamma\in\ma_\mathcal{C}(\bmu)$ such that:
\begin{itemize}
\item[\point\label{item1:pt:lag}] Inequality \eqref{eq:a_superieur_a_e}:
$\A(\Gamma)\geqslant \EE(\bmu)$
is an equality,
\item[\point] the measure $\Gamma$ is Markov,
\item[\point] it is the limit in $\p(\mathcal{C})$ of a sequence $(\como_{[R_n]})_n$.
\end{itemize}
Such a $\Gamma$ is unique in $\ma_\mathcal{C}(\bmu)$; it is the Markov-quantile process $\mq$.
\end{them}

\begin{proof} Proposition \ref{pro:ap_egale_emu} shows that $\Gamma=\mq$ satisfies (a)--(c). Theorem \ref{them:from} implies uniqueness from (b) and (c). 
\end{proof}

\noindent To state our second result, Theorem \ref{them:commun_d_et_1}, we need to introduce the following definition. In it, remember that an optimal transport plan is defined in Reminder \ref{reminder:wasserstein}.

\begin{defi}\label{defi:disp}\maz
Let $R=\{r_0,r_1,\ldots,r_m,r_{m+1}\}$ be a partition of $[0,1]$ and $\bmu=(\mu_t)_{t\in [0,1]}\in \p(\R^d)^{[0,1]}$. We denote by $\disp_{R}(\bmu)$ or more simply $\disp_{R}$ the set of measures $\Gamma\in\p(\mathcal{C}([0,1],\R^d))$ such that: (i) conditionally on any `present' time $r\in R$, the past is independent from the future; (ii) $\Gamma$ interpolates linearly (hence in fact optimally)  $\mu_{r_i}$ and $\mu_{r_{i+1}}$. The  conditions for $\Gamma$ to be in $\disp_{[R]}$ are more concretely the following: for each $i\in\{0,\ldots, m\}$,

\point\label{p1defi:disp} the coupling $\Gamma^{r_i,r_{i+1}}\in\ma(\mu_{r_i},\mu_{r_{i+1}})$ is an optimal transport plan between $\mu_{r_i}$ and $\mu_{r_{i+1}}$,

\point\ for $\{\lambda_1,\ldots,\lambda_n\}\subset[0,1]$ and $\mathrm{m}^\lambda:(x,y)\in (\R^d)^2\mapsto \lambda y+(1-\lambda)x$, we have:
\[(\mathrm{m}^{\lambda_1},\ldots, \mathrm{m}^{\lambda_n})_\#\Gamma^{r_i,r_{i+1}}=\Gamma^{\mathrm{m}^{\lambda_1}(r_i,r_{i+1}),\ldots, \mathrm{m}^{\lambda_n}(r_i,r_{i+1})},\]

\point\ for all finite $S=\{s^0_1,\ldots, s^0_{n_0},r_1,s^1_{1},\ldots,s^1_{n_1},r_2,\ldots,r_m,s^m_{1},\ldots,s^m_{n_m}\}$ containing $R\setminus\{r_0,r_{m+1}\}=\{r_1,\ldots,r_m\}$, 
$$(\pr^{S})_\# \Gamma= \Gamma^{s^0_1,\ldots, s^0_{n_0},r_1}\circ \Gamma^{r_1,s^1_{1},\ldots,s^1_{n_1},r_2}\circ\ldots\circ \Gamma^{r_m,s^m_{1},\ldots,s^m_{n_m}},$$
where the first and/or last terms disappear if $n_0$ and/or $n_m$ is null. 
\end{defi}

\begin{rem}\label{rem:singleton}Note that $\#\disp_{R}=1$ if and only if each set $\ma(\mu_{r_i},\mu_{r_{i+1}})$, appearing in \ref{p1defi:disp}, contains a unique optimal transport. It is the case when $d=1$, where $\ma(\mu_{r_i},\mu_{r_{i+1}})=\{\como(\mu_{r_i},\mu_{r_{i+1}})\}$, see Reminder \ref{reminder:wasserstein}.
\end{rem}

\begin{them}\label{them:commun_d_et_1}
Let $d$ be a positive integer and $\bmu=(\mu_t)_{t\in[0,1]}$ a curve of finite energy in $\p_2(\R^d)$. For every nested ({\em i.e.}, increasing) sequence $(R_n)_n$ of finite subsets $R_n$ of $[0,1]$, with $R_\infty\eqdef\cup_n R_n$ dense in $[0,1]$, and $\Gamma_n\in \disp_{R_n}$ for all $n\in \N$, there exists $\Gamma\in \ma_{\mathcal{C}}(\bmu)$ that is the limit in $\p(\mathcal{C}([0,1],\R^d))$ of a subsequence of\/ $(\Gamma_n)_n$. Moreover for every $\Gamma$ obtained in this way the action $\A(\Gamma)$ is minimal, {\em i.e.}, such that Inequality \eqref{eq:a_superieur_a_e} is an equality.

Moreover, in dimension $d=1$, a Markov limit $\Gamma$ exists and if a limit $\Gamma$ is Markov, it is the Markov-quantile measure in $\ma_\mathcal{C}((\mu_t)_{t\in [0,1]})$.
\end{them}

\begin{proof}\label{proof:them:commun_d_et_1}
Adapting \cite[Chapter 7]{Vi2} (written in the spirit of \cite{BeBu}), \cite{Lis} or Proposition \ref{pro:ap_egale_emu} to our context we obtain the first part of the theorem for every $d\geq 1$. This requires slight modifications that we do not detail: Villani's chapter is in fact written for geodesic curves $(\mu_t)_t$ between prescribed $\mu_0$ and $\mu_1$ whereas Lisini's processes are attached to curves $(\mu_t)_{t\in [0,1]}$ of finite energy but the processes of the sequence are constant on each interval between two consecutive points of the partition, whereas ours is linear. Note, as an indication, that our measures $\Gamma_n$ minimize ${\cal A}$ in $\{\Gamma\in \p(\mathcal{C}([0,1],\R^d)):\,\forall r\in R_n,\,\Gamma^r=\mu_r\}$, the minimum being $\A(\Gamma_n)=\EE(\bmu,R_n)$.

In case $d=1$, take the nested sequence $(R_n)_n$ given by Theorem \ref{them:from}, then $\como_{[R_n]}$ converges to $\mq$ in $\p(\C)$ by Proposition \ref{pro:ap_egale_emu}. Up to taking a subsequence, the same sequence of partitions permits $\Gamma_n\in\disp_{R_n}$ to converge to some $\Gamma$. By Definitions \ref{defi:quantile_discretement_markovien} and \ref{defi:disp}, for every $S\subset R_n$ the measure $(\pr^S)_\#\Gamma_n$ coincides with $(\pr^S)_\#\como_{[R_n]}$ so that
\[
(\pr^S)_\#\Gamma=(\pr^S)_\#\mq.
\]
As $R_\infty$ is dense in $[0,1]$ and the measures are concentrated on $\mathcal{C}$ we have $\Gamma=\mq$. This proves the existence for $d=1$.

To establish uniqueness, take as before a nested sequence $(R_n)_n$ and let $\Gamma_n$ be the single element of $\disp_{R_n}$ (see Remark \ref{rem:singleton}). Assume that $(\Gamma_n)_n$ has a Markov limit $\Gamma$. By Definitions \ref{defi:quantile_discretement_markovien} and \ref{defi:disp}, for every $S\subset R_n$ the measure $(\pr^S)_\#\Gamma_n$ coincides with $(\pr^S)_\#\como_{[R_n]}$. Using the same argument as for Proposition \ref{pro:ap_egale_emu}, up to taking a subsequence, $(\como_{[R_n]})_n$ converges to an element of $\ma_\C(\bmu)$ that we denote by $\Gamma'$. Hence for every $S\subset R_\infty$, $
(\pr^S)_\#\Gamma=(\pr^S)_\#\Gamma'.$ As $R_\infty$ is dense in $[0,1]$ and $\Gamma$, $\Gamma'$ are concentrated on $\mathcal{C}$ we have $\Gamma'=\Gamma$. Therefore $\Gamma'$ is Markov. Uniqueness in Theorem \ref{them:from} implies $\Gamma'=\mq$. Thus $\mq$ is the unique possible Markov limit for $(\Gamma_n)_n$.
\end{proof}

\begin{rem}\label{rem:lisini}
Our work differs from Lisini's paper \cite{Lis} in several points. In Theorem \ref{them:commun_d_et_1}: (i) we restrict the range of application to $\mathcal X=\R^d$, (ii) our interpolations are continuous and piecewise linear instead of piecewise constant, (iii) we consider the uniform distance between the curves and the resulting weak convergence, instead of the weak topology on $L^p([0,1],{\mathcal X})$, (iv)  our partitions are adapted in order to ensure (in case $d=1$) the Markov property at the limit while the partitions in \cite{Lis} are dyadic.
\end{rem}

\section{Open questions: a Markov minimizer for the action in metric spaces}\label{sec:unique}

Let us finish by mentioning possible connection of our theorems with a stream of research whose latest developments are to be found in the so-called Brenier--Schrödinger problem (see for instance the works by Arnaudon \emph{et al.} \cite{ACLZ}, Benamou, Carlier and Nenna \cite{BeCaNe}, Baradat and L\'eonard \cite{BaLe}, and the references therein). In this modified problem the trajectories become diffusion trajectories with drift and the new setting comes together with a natural action functional for the quantification of large deviations. It corresponds to an entropic minimization problem over the flows (the name given there for $\Gamma$) with marginals prescribed at any times (in the basic problem, the same measure for every $t\in [0,1]$) and prescribed joint law between the terminal measures. As a referee pointed out to us the situation is even closer to the setting studied by Dawson and Gärtner \cite{Dawson1987} where, as in our situation, the last condition is not prescribed. Since the minimizer of the entropy is Markov (see \cite[Section 1.4 of Chapter II]{Foellmer1988}) it is tempting to figure out that some alternative approach could exist for constructing the Markov-quantile process. However, until now we failed to create this connection, one major obstruction being that the measures $\mu_t$ in the family $\bmu=(\mu_t)_{t\in [0,1]}$ apparently have to be diffuse, another related fundamental obstruction being the non stability of the Markov property for the family of processes attached to a one-parameter family of mollifiers $(\bmu^\eps)_{\eps>0}$.

We gather here the main questions arising in the paper.\maz

\point\ We proved that choosing the sequence $(R_n)_n$ properly, the approach introduced by Lisini to build Lagrangian representatives converges in dimension $d = 1$ towards a Markov process, so that there exists a Markov minimal Lagrangian representative. Is it still true in higher dimension? In geodesic spaces? Also, in dimension $d = 1$, we saw that there is only one possible Markov limit for this approach, namely the Markov-quantile process. Can also this be generalized?

\point\ Can $\mq$ or more general objects in Polish spaces be equivalently introduced through a large deviation approach inspired by the Schrödinger problem? See the paragraph just before.%end of Subsection \ref{subsec:alt_appr}.

\point\ Other questions are listed in \S5 of our first paper \cite{boujui1}. Is for instance $\mq$ a strongly Markov process? Example \ref{ex:MLr_nu} shows that the simple Markov property fails to uniquely determine $\mq$ {among minimal Lagrangian representatives. Can it be characterized by a more stringent stochastic property?

\section*{Aknowledgement}

We thank the referees, especially one of them for their in-depth work leading to this very amended version.

\bibliographystyle{imsart-number}
\bibliography{basebib_quantiles3}

\end{document}